\documentclass[10pt,reqno]{amsart}

\newcommand\version{October 20, 2010}

\usepackage{amsmath,amsfonts,amsthm,amssymb,amsxtra,dsfont}

\newtheorem{theorem}{Theorem}[section]
\newtheorem{proposition}[theorem]{Proposition}
\newtheorem{lemma}[theorem]{Lemma}

\theoremstyle{definition}

\newtheorem{remark}{Remark}

\numberwithin{equation}{section}

\newcommand{\C}{\mathbb{C}}

\renewcommand{\epsilon}{\varepsilon}

\newcommand{\N}{\mathbb{N}}

\newcommand{\R}{\mathbb{R}}

\newcommand{\Sph}{\mathbb{S}}

\DeclareMathOperator{\Div}{div}

\DeclareMathOperator{\im}{Im}

\DeclareMathOperator{\spa}{span}
\DeclareMathOperator{\spec}{spec}

\renewcommand{\leq}{\leqslant}
\renewcommand{\geq}{\geqslant}

\renewcommand{\to}{\rightarrow}

\begin{document}

\title[Ground states for boson star equation --- \version]{On ground states for \\ the  $L^2$-critical boson star equation}

\author{Rupert L. Frank}
\address{Rupert L. Frank, Department of Mathematics,
Princeton University, Washington Road, Princeton, NJ 08544, USA}
\email{rlfrank@math.princeton.edu}

\author{Enno Lenzmann}
\address{Enno Lenzmann, Department of Mathematical Sciences, University of Copenhagen, Universitetspark 5, 2100 Copenhagen \O, Denmark}
\email{lenzmann@math.ku.dk}

\thanks{\copyright\, 2010 by the authors. This paper may be reproduced, in its entirety, for non-commercial purposes.}

\begin{abstract}
We consider ground state solutions $u \geq 0$ for the $L^2$-critical boson star equation 
$$
\sqrt{-\Delta} \, u - \big ( |x|^{-1} \ast |u|^2 \big ) u = -u \quad \mbox{in $\R^3$}.
$$
We prove analyticity and radial symmetry of $u$.

In a previous version of this paper, we also stated uniqueness and nondegeneracy of ground states for the $L^2$-critical boson star equation in $\R^3$, but the arguments given there contained a gap. However, we refer to our recent preprint \cite{FraLe} in {\tt arXiv:1009.4042}, where we prove a general uniqueness and nondegeneracy result for ground states of nonlinear equations with fractional Laplacians in $d=1$ space dimension.
\end{abstract}

\maketitle

\section{Introduction and Main Results}


We consider ground states for the {\em massless boson star equation} in $d=3$ dimensions given by
\begin{equation}
\label{eq:eq}
\left \{ \begin{array}{l} \displaystyle \sqrt{-\Delta} \, u - \big ( |x|^{-1} \ast |u|^2 \big ) u   =  -u ,  \\[1ex]
u \in H^{1/2}(\R^3), \quad u \geq 0, \quad u \not \equiv 0 .
\end{array} \right .
\end{equation}
Here $H^s(\R^3)$ is the inhomogeneous Sobolev space of order $s \in \R$, and the symbol $\ast$ denotes the convolution on $\R^3$. 

The nonlinear equation \eqref{eq:eq} plays a central role in the mathematical theory of gravitational collapse of boson stars, which we briefly summarize as follows. In the seminal work of Lieb and Yau \cite{LiYa}, the universal constant
\begin{equation}
\label{eq:Nstar}
N_* = \| u \|_{2}^2
\end{equation}
was found to be the so-called {\em ``Chandrasekhar limiting mass''} for boson stars in the time-independent setting. Here the ground state $u \in H^{1/2}(\R^3)$, appearing in equation \eqref{eq:Nstar}, is a certain optimizer that solves problem \eqref{eq:eq}. As one main result, it was proven in \cite{LiYa} that boson stars with total mass strictly less than $N_*$ are gravitationally stable, whereas boson stars whose total mass exceed $N_*$ may undergo a ``gravitational collapse'' based on variational arguments and many-body quantum theory. Moreover, it was conjectured by Lieb and Yau in \cite{LiYa} as an open problem that uniqueness for ground states holds.

More recently, the mathematical theory of boson stars has entered the field of nonlinear dispersive equations: In \cite{ElSc}, it was shown that the {\em dynamical evolution} of boson stars is effectively described by the nonlinear evolution equation (with mass parameter $m \geq 0$)
\begin{equation}
\label{eq:boson}
i \partial_t \psi = \sqrt{-\Delta + m^2} \, \psi - \big ( |x|^{-1} \ast |\psi|^2 \big ) \psi
\end{equation}
for the wave field $\psi : [0,T) \times \R^3 \to \C$. In fact, this dispersive nonlinear $L_2$-critical PDE displays a rich variety of phenomena such as stable/unstable traveling solitary waves and finite-time blowup. In particular, the ground states $u(x) \geq 0$ for \eqref{eq:eq} and the constant $N_* >0$ given by \eqref{eq:Nstar} both play a fundamental role as follows: First, the ground state solutions $u(x)$ of \eqref{eq:eq} give rise to {\em ground state solitary waves} of the form
\begin{equation}
\psi(t,x) = e^{it} u(x)
\end{equation}
for the evolution equation \eqref{eq:boson} in the case of vanishing mass $m=0$. Second, the universal constant $N_* > 0$ sets the scale between ``small'' and ``large'' solutions of the $L_2$-critical nonlinear dispersive PDE \eqref{eq:boson}, irrespectively of the value for $m \geq 0$. More precisely, as shown in \cite{FrLe,Le2}, all solutions $\psi \in C_0^t H^{1/2}_x([0,T) \times \R^3)$ with small $L_2$-mass 
$$\| \psi(t) \|_{2}^2 < N_*$$ 
extend globally in time (i.\,e. we have $T=\infty$); whereas solutions with 
$$\| \psi(t) \|_{2}^2 > N_*$$ 
can lead to blowup at some finite time $T < \infty$. (This singularity formation indicates the dynamical  ``gravitational collapse'' of a boson star.) Thus, any analytical insight into some key properties (e.\,g., uniqueness up to translation) of the ground states $u(x) \geq 0$ and the spectrum of their linearization will be of considerable importance for a detailed blowup analysis for the nonlinear dispersive equation \eqref{eq:boson}.

\medskip
Our main result is as follows.

\begin{theorem}\label{main} {\bf (Radiality and Analyticity).}
Every solution $u \in H^{1/2}(\R^3)$ of problem \eqref{eq:eq} is of the form $u(x)=Q(x-a)$ for some $a \in \R^3$, where $Q$ satisfies the following properties.

\begin{enumerate}
\item[(i)] \label{eq:symmdecr} $Q$ is positive, radial and strictly decreasing.
\item[(ii)] $Q$ is real-analytic. More precisely, there exists a constant $\sigma>0$ and an analytic function $\tilde Q$ on $\{z\in\C^3 : |\im z_j|<\sigma, 1\leq j\leq 3 \}$ such that $\tilde Q(x)=Q(x)$ if $x\in\R^3$.
\end{enumerate}
\end{theorem}


\begin{remark}
A natural open question is uniqueness of the ground state $Q=Q(|x|) > 0$. We refer to our recent work \cite{FraLe}, where uniqueness has been proven for ground state of nonlinear equations with fractional Laplacians $(-\Delta)^s$ in $d=1$ dimension.
\end{remark}

\begin{remark}
Our proof that any solution of problem \eqref{eq:eq} must be radially symmetric (with respect to some point) employs the classical method of moving planes introduced in \cite{GiNiNi}; see Section~\ref{sec:symmetry} below. See also \cite{BiLoWa} for a similar symmetry result for the moving plane method applied to equation with fractional Laplacians on the unit ball $\{ x \in \R^3 : |x| < 1\}$. We remark that the arguments, which we present in Section \ref{sec:symmetry} below, are able to deal with the unbounded domain $\R^3$, and thus settling an open problem stated in \cite{BiLoWa}.

While finalizing the present paper, we learned that \cite{MaZh} have very recently and independently established a symmetry result for the equation $-\Delta \, u - \big ( |x|^{-1} \ast |u|^2 \big ) u   =  -u$ in $\R^3$. They also briefly sketch \cite[Sec. 5]{MaZh} how to extend their approach to more general equations, including \eqref{eq:eq}. Their method is different from ours and uses the integral version of the method of moving planes developed in \cite{ChLiOu}. We believe that our non-local Hopf's lemma, on which our differential version of moving planes is based, might have applications beyond the context of the present paper.
\end{remark}

\begin{remark} \label{rem:rescale}
Note that Theorem \ref{main} implies an analagous statement for positive solutions of the equation 
$$\sqrt{-\Delta}\, u - \kappa(u^2*|x|^{-1}) u = -\lambda u.$$ 
with constants $\kappa,\lambda>0$. Indeed, $u$ solves this equation if and only if $\kappa^{-1/2} \lambda^{-3/2} u(x/\lambda)$ solves \eqref{eq:eq}. One might also ask whether this equation can have a solution for $-\lambda=E\geq0$. The answer is negative, even if the positivity assumption of $u$ is dropped, as shown by the next result (whose proof is given in Subsection \ref{sec:nonexist} below). Without loss of generality,  we can put  $\kappa=1$ in the following. 
\end{remark}

\begin{proposition}\label{nonexist}
Let $E\geq0$. If $u\in H^{1/2}(\R^3)$ is radial and satisfies $\sqrt{-\Delta} \, u - (|u|^2 * |x|^{-1})u = Eu$, then $u\equiv 0$.
\end{proposition}

\subsection*{Organization of the Paper}
In Sections \ref{sec:prelim}--\ref{sec:anal}, we organize the proof of Theorem \ref{main} as follows. In Section \ref{sec:prelim} we collect some preliminary results on \eqref{eq:eq} about the existence, regularity, and spatial decay of solutions. Moreover, we give the proof of Proposition \ref{nonexist}. In Section \ref{sec:symmetry} we implement the method of moving planes and we prove that every solution of \eqref{eq:eq} is spherically symmetric with respect to some point. In Section \ref{sec:anal} we prove the real analyticity of solutions. In Section \ref{sec:anal2}, we provide further analyticity results about elements in the kernel of the linearization of \eqref{eq:eq}.

\subsection*{Notation and Conventions}
For the Fourier transform on $\R^3$, we use the convention
\begin{equation}
\label{eq:fourier}
\hat u(\xi) := (2\pi)^{-3/2} \int_{\R^3} u(x) e^{-i\xi\cdot x} \,dx .
\end{equation}
As usual, the fractional derivative operators $(-\Delta)^s$ and $(1-\Delta)^s$ are defined via their multipliers $|\xi|^{2s}$ and $(1+|\xi|^2)^s$ in Fourier space, respectively. Lebesgue spaces of functions on $\R^3$ will be denoted by $L_p = L_p(\R^3)$ with norm $\| \cdot \|_{p}$ and $1 \leq p \leq \infty$. For the sake of brevity, we shall use the notation $\|u \| \equiv \| u \|_2$ occasionally. We employ inhomogeneous Sobolev norms $\| u \|_{H^s} := \| (1-\Delta)^{s/2} u \|_2$, as well as homogeneous Sobolev norms $\| u \|_{\dot{H}^s} = \| (-\Delta)^{s/2} u \|_2$. The equation \eqref{eq:eq} is always understood to hold in the $H^{-1/2}$ sense. That is, we say that $u \in H^{1/2}(\R^3)$ solves the equation in \eqref{eq:eq} if
$$
\int_{\R^3} |\xi| \overline{\hat v(\xi)} \hat u(\xi) \,d\xi - \iint_{\R^3\times\R^3} \frac{\overline{v(x)} u(x) |u(y)|^2}{|x-y|} \,dx\,dy = - \int_{\R^3} \overline{v(x)} u(x) \,dx \,, 
$$
for all $v \in H^{1/2}(\R^3)$.

In what follows, the letter $C$ denotes a constant which is allowed to change from inequality to inequality. With the usual abuse of notation, we shall not distinguish between the functions $f(|x|)$ and $f(x)$ whenever $f : \R^3 \rightarrow \C$ is radial.

\subsection*{Acknowledgments}
R.\,F. gratefully acknowledges support through DFG grant FR 2664/1-1 and NSF grant PHY 06 52854. E.\,L.~is supported by a Steno fellowship of the Danish research council and NSF grant DMS-0702492.


\section{Preliminary results} \label{sec:prelim}

To prepare the proof of our main results, we first collect some preliminary results on the existence, regularity, and decay of solutions to problem \eqref{eq:eq}. Since all these facts follow from the literature and standard arguments, we will keep our exposition brief throughout this section. 

\subsection{Existence and properties of a minimizing solution}
The existence of a nonnegative, radial solution $Q(|x|) \geq 0$ of problem \eqref{eq:eq} can be established by direct variational arguments, as remarked in \cite[App. A.2]{LiYa}. More precisely, we consider the minimization problem
\begin{equation} \label{eq:varprob}
\inf \big \{ I[u]:  u\in H^{1/2}(\R^3), \, u \not \equiv 0 \big \}  ,
\end{equation}
where
\begin{equation}
I[u] := \frac{\|(-\Delta)^{1/4} u\|^2 \ \|u\|^2}{  \iint_{\R^3 \times \R^3} |u(x)|^2 |x-y|^{-1} |u(y)|^2 \, dx\,dy} \,.
\end{equation}
Thanks to strict rearrangement inequalities (see \cite{Li,FrSe}), we have that $I[u^*] \leq I[u]$ with equality if and only if $u(x)$ equals its symmetric-decreasing rearrangement $u^*(|x|)\geq 0$ (modulo translation in space and multiplication by a complex number). As pointed out in \cite[App. A.2]{LiYa}, this fact permits us to imitate the proof in \cite{LiOx} to deduce the existence of a symmetric-decreasing minimizer $Q = Q^* \in H^{1/2}(\R^3)$ for problem (\ref{eq:varprob}). Moreover, an elementary calculation shows that any minimizer for \eqref{eq:varprob} satisfies the Euler-Lagrange equation
$$
\sqrt{-\Delta} \, Q - \kappa(|Q|^2*|x|^{-1}) Q = -\lambda Q,
$$
with some constants $\lambda >0$ and $\kappa>0$ that both depend on $Q$. By Remark \ref{rem:rescale}, we see that any symmetric-decreasing minimizer $Q = Q^* \in H^{1/2}(\R^3)$ for \eqref{eq:varprob} furnishes (after suitable rescaling) a solution of problem \eqref{eq:eq}.

\subsection{Regularity and Decay}

In this subsection, we collect some basic regularity and decay estimates for solutions $u \in H^{1/2}(\R^3)$ of the nonlinear equation
\begin{equation} \label{eq:eq1}
\sqrt{-\Delta} \, u - \big ( |u|^2 \ast |x|^{-1} \big ) u =  -u .
\end{equation}
Note that we do not require $u$ to be non-negative or even real-valued in this section, unless we explicitly say so.

\begin{lemma}[Smoothness of solutions]
\label{smooth}
Let $u\in H^{1/2}(\R^3)$ be a solution of \eqref{eq:eq1}. Then $u \in H^s(\R^3)$ for all $s \geq 1/2$.
\end{lemma}

\begin{proof}
This follows from a simple bootstrap argument. Indeed, note that $u$ satisfies
\begin{equation} \label{eq:uresolvent}
u = (\sqrt{-\Delta} + 1)^{-1} F(u),
\end{equation}
where we put $F(u) = (|u|^2 \ast |x|^{-1}) u$. Since $F(u)$ maps $H^s(\R^3)$ into itself for any $s \geq 1/2$ (see, e.\,g., \cite{Le2})  and thanks to the smoothing property $(\sqrt{-\Delta} + 1)^{-1} : H^s(\R^3) \rightarrow H^{s+1}(\R^3)$, we obtain the desired result. 
\end{proof}

Next, we record a decay estimate for solutions of \eqref{eq:eq1}. 

\begin{lemma}[Decay rates]
\label{decay}
Any solution $u \in H^{1/2}(\R^3)$ of \eqref{eq:eq1} satisfies
\begin{equation}
\label{ineq:upper}
|u(x)| \leq C (1 + |x|)^{-4} 
\end{equation}
and
\begin{equation}
\label{ineq:newton}
(|u|^2 \ast |x|^{-1})(x) \leq C (1 + |x|)^{-1} .
\end{equation}
Moreover, if we assume that $u(x) \geq 0$ and $u \not \equiv 0$, then we also have the lower bound
\begin{equation}
\label{ineq:lower}
u(x) \geq C (1+|x|)^{-4}.
\end{equation} 
In particular, any such solution $u(x)$ is strictly positive.
\end{lemma}

\begin{proof}
Note that $u \in L_2(\R^3)$ is an eigenfunction for the ``relativistic'' Schr\"odinger operator
$$
H := \sqrt{-\Delta} + V
$$
with the local potential $V(x) = -(|u|^2 \ast |x|^{-1})(x)$. Furthermore, by Lemma \ref{smooth} and Sobolev embeddings, we have $u \in L_p(\R^3)$ for all $p \geq 2$, which implies that $V(x)$ is continuous and $V(x) \rightarrow 0$ as $|x| \rightarrow \infty$. Hence $u$ is an eigenfunction corresponding to the eigenvalue $-1$ below the bottom of the essential spectrum of $H$. From \cite[Proposition IV.1]{CaMaSi} we now deduce the bound \eqref{ineq:upper}.

Next, we see that deriving the bound \eqref{ineq:newton} amounts to estimating the function $V(x)$ defined above. First, we note that the Hardy-Kato inequality (see, e.\,g., \cite{He}) gives us
$$
| V(x) | \leq \sup_{x \in \R^3} \int_{\R^3} \frac{|u(y)|^2}{|x-y|} \, dy \leq C \int_{\R^3} \overline{u}(y) (\sqrt{-\Delta} \, u)(y) \,dy  \leq C \| u \|_{H^{1/2}}^2.
$$
Also, from \eqref{ineq:upper} we have a radially symmetric bound for $u(x)$. Thus, by Newton's theorem (see, e.\,g., \cite[Theorem 9.7]{LiLo}), we deduce
$$
| V(x) | \leq \frac{1}{|x|} \int_{\R^3} \frac{C}{(1+|y|)^4} \, dy \leq \frac{C}{|x|} .
$$   
Combining these two estimates for $V(x)$, we obtain the bound \eqref{ineq:newton}.

Finally, let us also assume that $u(x) \geq 0$ and $u \not \equiv 0$. By standard Perron-Frobenius arguments, we conclude that $u(x)$ is the unique ground state eigenfunction for the Schr\"odinger operator $H$. In particular, invoking \cite[Proposition IV.3]{CaMaSi} yields the lower bound \eqref{ineq:lower}. \end{proof}

\begin{remark}
As an alternative to probabilistic arguments used in \cite{CaMaSi}, we could also provide a more ``hands-on'' proof of Lemma \ref{decay}, which is based on bootstrapping equation \eqref{eq:uresolvent} and using the explicit formula for the Green's function 
\begin{align}
(\sqrt{-\Delta} + \tau )^{-1}(x,y) & = \int_0^\infty e^{-t \tau} \exp (-t \sqrt{-\Delta} )(x,y) \, dt \nonumber \\ & = \frac{1}{\pi^2} \int_0^\infty e^{-t \tau} \frac{t}{(t^2 + |x-y|^2)^2} \, dt . \label{eq:reskernel}
\end{align}  
We refer to \cite{Le3} for details for this alternate proof; see, e.\,g., \cite[Section 7.11]{LiLo} for the explicit formula of the kernel.
\end{remark}


\subsection{Proof of Proposition \ref{nonexist}}\label{sec:nonexist}

Suppose $u \in H^{1/2}(\R^3)$ is radial and solves 
$$
\sqrt{-\Delta} \, u - \big (|x|^{-1} \ast |u|^2 \big ) u = E u
$$
with some constant $E \geq 0$. With $V:= - |u|^2 * |x|^{-1}$ one has the virial identity
$$
\|(-\Delta)^{1/4} u\|^2 = \int |x| \partial_r V |u|^2 \,dx \,,
$$
which can be proved along the lines of \cite[Thm. 4.21]{Th}. (The assumptions on $V$ follow easily from Newton's theorem.) Next, integrating the equation against $u$ shows that $\|(-\Delta)^{1/4} u\|^2 + \int V |u|^2 \,dx = E\|u\|^2$. Hence,
$$
\int_{\R^3} (V+|x| \partial_r V -E) |u|^2 \,dx =0 \,.
$$
But Newton's theorem gives us
$$
V(x) + |x| \partial_r V(x) = -4\pi \int_r^\infty |u(s)|^2 s\,ds \leq 0 \,,
$$
Therefore we have $(V+|x| \partial_r V -E) |u|^2 = 0$ almost everywhere. If $E > 0$, this shows directly that $u \equiv 0$. If $E=0$ holds, then we conclude $(\int_r^\infty |u(s)|^2 s \,ds ) u(r) = 0$ for almost every $r \geq 0$, which again implies $u \equiv 0$. This completes the proof of Proposition \ref{nonexist}. \hfill $\blacksquare$


\section{Symmetry} \label{sec:symmetry}

We now establish our first main result of Theorem \ref{main}. That is, any nonnegative solution $u(x) \geq 0$ of problem \eqref{eq:eq} is radially symmetric up to translation. The basic strategy rests on the method of moving planes, which was applied in \cite{GiNiNi2} to obtain a similar statement for the local elliptic equations of the form $-\Delta u + f(u) = 0$. To make the method of moving planes work successfully in our case, we establish a suitable ``non-local Hopf lemma'' below.

The goal of this section is to establish the following result.
\begin{theorem}[Symmetry]\label{symm}
Any solution of problem \eqref{eq:eq} is radial with respect to some point and strictly decreasing with respect to the distance from that point.
\end{theorem}

Since radial symmetry around a point means reflection symmetry with respect to any plane passing through that point, we start by proving a result about reflections. For the sake of concreteness, we consider reflections on the plane $\{x_1=0\}$. The following assertion will immediately imply Theorem \ref{symm}.

\begin{proposition}\label{symm1}
Let $u \in H^{1/2}(\R^3)$ be a solution of problem \eqref{eq:eq} and assume that the function $f := (u^2*|x|^{-1}) u$ satisfies
\begin{equation}
\label{eq:com}
\int_{\R^3} y_1 f(y) \,dy = 0 \,.
\end{equation}
Then, for each $x'\in\R^2$ fixed, the function $u(\cdot,x')$ is symmetric with respect to the point $x_1=0$ and strictly decreasing for $x_1>0$.
\end{proposition}

Before we turn to the proof of Proposition \ref{symm1}, we first give the proof of Theorem \ref{symm}.

\begin{proof}[Proof of Theorem \ref{symm} assuming Proposition \ref{symm1}]
Let $u$ be a solution of problem \eqref{eq:eq} and define $f := (u^2*|x|^{-1}) u$. Since $u \geq 0$ and $u \not \equiv 0$, we have $\int_{\R^3} f(y) \, dy > 0$. Thus there exists a translation $a\in\R^3$ such that 
\begin{equation}
\label{eq:comall}
\int_{\R^3} y_j f(y-a) \,dy = 0, 
\qquad
\text{for}\ j=1,2,3 \,.
\end{equation}
(We note that the integrals converge absolutely in view of the estimates from Lemma \ref{decay}.) For any orthogonal matrix $R\in O(3)$, the function $v_R(x) := u(Rx-a)$ is a solution of \eqref{eq:eq} and the normalization \eqref{eq:comall} implies that $f_R:= (v_R^2*|x|^{-1}) v_R$ satisfies \eqref{eq:com}. Hence, by Proposition \ref{symm1}, the function $v_R(x)$ is symmetric with respect to $x_1=0$ and strictly decreasing for $x_1>0$. Since the rotation $R \in O(3)$ is arbitrary, this means that $u$ is radial with respect to $a$ and strictly decreasing as a function of $|x-a|$.
\end{proof}

The proof of Proposition \ref{symm1} will be given in Subsection \ref{sec:symm1proof} after having proved two preliminary results. In this section we use the following notation. For any $\lambda\in\R$ and any point $x=(x_1,x') \in \R \times \R^2$, we denote by 
\begin{equation}
\label{eq:xrefl}
x^\lambda := (2\lambda-x_1,x')
\end{equation}
its reflection with respect to the hyperplane $\{x_1=\lambda\}$. Moroever, the reflection of a function $u$ on $\R^3$ with respect to the hyperplane $\{ x_1=\lambda\}$ will be denoted by 
\begin{equation}
\label{eq:urefl}
u_\lambda(x) := u(x^\lambda). 
\end{equation}


\subsection{Asymptotics of the solution}\label{sec:asymp}

Recall from Lemma \ref{decay} that any solution $u$ of \eqref{eq:eq} decays like $|x|^{-4}$. To make the method of moving planes work, we need more precise asymptotics of $u$ and its first derivative. To this end, we consider the equations of the following general form: 
\begin{equation}\label{eq:asymp}
\sqrt{-\Delta} \, u =-u +f ,
\end{equation}
where the inhomogeneity $f(x)$ is some given measurable function on $\R^3$. Clearly, this equation coincides with equation in \eqref{eq:eq} if we put $f:=(u^2*|x|^{-1}) u$; and according to our a-priori estimates from Lemma \ref{decay}, we then have $0< f(x) \leq C (1+|x|)^{-5}$. In fact, our asymptotics will be valid for more general inhomogeneities $f(x)$. The precise statement is as follows.

\begin{lemma}\label{asymp}
Assume that $u\in H^1(\R^3)$ satisfies \eqref{eq:asymp} with $|f(x)| \leq C (1+|x|)^{-\rho}$ for some $\rho>4$. Then
\begin{enumerate}
\item[(i)] $\lim_{|x|\to\infty} |x|^4 u(x) = \pi^{-2} \int f(y) \,dy$.
\item[(ii)] $\lim_{x_1\to\infty} \frac{|x|^6}{x_1} \frac{\partial u}{\partial x_1}(x) = - 4 \pi^{-2} \int f(y) \,dy$.
\item[(iii)]
If $\lambda^j\to\lambda$ and $|x^j|\to\infty$ with $x_1^j<\lambda^j$, then
\begin{equation*}
\lim_{j\to\infty} \frac{|x^j|^6}{2(\lambda^j -x_1^j)} \left(u(x^j)-u_{\lambda^j}(x^j)\right) = \frac4{\pi^2} \int_{\R^3} f(y) (\lambda-y_1) \,dy \,,
\end{equation*}
where $u_{\lambda^j}(x)$ is defined in \eqref{eq:urefl} above.
\end{enumerate}
\end{lemma}

\begin{proof}
We write $u=(\sqrt{-\Delta}+1)^{-1}f$ and use the explicit formula \eqref{eq:reskernel} for the resolvent kernel. Calculating the corresponding results for this kernel we easily obtain the assertion of the lemma for $f$'s with compact support. We omit the details. The extension to more general $f$'s uses a density argument in the same spirit as in \cite[Lem. 2.1]{GiNiNi2}.  \end{proof}


\subsection{A non-local Hopf lemma}

As a next step, we derive the following non-local Hopf lemma.

\begin{lemma}\label{hopf}
Let $w\in H^{1}(\R^3) \cap C^1(\R^3)$ be odd with respect to the plane $\{x_1=0\}$ and assume that, for some $\tau \in \R$, we have
\begin{equation}
\label{eq:hopf}
\begin{split}
\sqrt{-\Delta}  \, w & \geq -\tau w
\qquad \text{in} \ \{x_1>0\} \,, \\
w & \geq 0
\qquad \text{in}  \ \{x_1>0\} \,.
\end{split}
\end{equation}
Then either $w\equiv 0$, or else $w>0$ in $\{x_1>0\}$ and $ \frac{\partial w}{\partial x_1} \big |_{x_1=0}>0$.
\end{lemma}

A different extension of Hopf's lemma to the non-local context is proved in \cite{BiLoWa}. Their approach does not allow for positive values of $\tau$ which, however, will be crucial for us.

\begin{proof}
Since we assume $w \geq 0$, it is sufficient to do the proof assuming that $\tau > 0$ holds. Next, we assume that $w\not\equiv 0$ and define $h:=(\sqrt{-\Delta} +\tau)w$. We note that $h$ is odd with respect to the plane $\{x_1=0\}$ and that $h\geq 0$ in $\{x_1>0\}$. Moreover, one easily sees that $h\not\equiv 0$; e.\,g.~via the Fourier transform. Next, we write
$$
w=(\sqrt{-\Delta} +\tau)^{-1}h = \int_0^\infty e^{- t\tau} \exp(-t\sqrt{-\Delta})h  \,dt \, .
$$
This shows that it is enough to prove that $\exp(-t\sqrt{-\Delta})h$ is strictly positive in $\{x_1>0\}$ and has a strictly positive $x_1$-derivative on $\{x_1=0\}$.

Using that $h$ is odd with respect to the plane $\{ x_1 = 0\}$ and writing $x=(x_1,x')$, $y=(y_1,y')$, we find (recalling formula \eqref{eq:reskernel} for the integral kernel) that
\begin{align*}
\exp(-t\sqrt{-\Delta})h (x) & = \pi^{-2} \int_{\R^3} \frac{t}{(t^2+|x-y|^2)^2} h(y) \,dy \\
& = \pi^{-2} \int_{y_1>0} \left(\frac{t}{(t^2+|x-y|^2)^2}  \right .\\
& \quad \left .  - \frac{t}{(t^2+(x_1+y_1)^2+|x'-y'|^2)^2} \right) h(y) \,dy \,.
\end{align*}
If $x_1>0$, then the integrand is non-negative and $\not \equiv 0$, and hence $\exp(-t\sqrt{-\Delta})h (x)>0$. Differentiating the above expression under the integral sign (which can be justified by dominated convergence), we find
$$
\frac{\partial}{\partial x_1} \exp(-t\sqrt{-\Delta})h (0,x')
= \frac{8}{\pi^{2}} \int_{y_1>0} \frac{t y_1}{(t^2+y_1^2+|x'-y'|^2)^3} h(y) \,dy \,,
$$
which again is strictly positive. \end{proof}


\subsection{Proof of Proposition \ref{symm1}}
\label{sec:symm1proof}

Now we are ready to implement the method of moving planes. Let $u$ be a solution of problem \eqref{eq:eq} and assume that $f := (u^2*|x|^{-1}) u$ satisfies \eqref{eq:com}.
Recalling the definition of $u_\lambda$ before Subsection \ref{sec:asymp}, we define the set
$$
\Lambda := \{ \mu>0 :\ \text{for all} \  \lambda > \mu \ \text{and for all}\ x \ \text{with}\ x_1 < \lambda \ \text{one has}\  u(x)\geq u_\lambda(x) \} \,.
$$
We divide the proof of Proposition \ref{symm1} into three steps as follows.

\emph{Step 1. $\Lambda$ is non-empty.}

We first note that according to Lemma \ref{asymp} (ii) there is a $\overline\lambda>0$ such that
\begin{equation}
\label{eq:decrease}
\frac{\partial u}{\partial x_1}(x) < 0
\qquad
\text{if}\ x_1 \geq \overline\lambda \,.
\end{equation}

We now prove that $\Lambda$ is non-empty by contradiction. If $\Lambda$ were empty, there would exist sequences of numbers $(\lambda^j)\to\infty$ and points $(x^j)$ with $x^j_1<\lambda^j$ such that
\begin{equation}
\label{eq:cont1}
u(x^j)<u_{\lambda^j}(x^j) \,.
\end{equation}
Next, we claim that
\begin{equation}
\label{eq:cont1inf}
|x^j|\to\infty
\end{equation}
and, with $\overline\lambda$ from \eqref{eq:decrease},
\begin{equation}
\label{eq:cont1bdd}
x^j_1 <\overline\lambda \,.
\end{equation}
To prove our claim, we note that $(x^j)^{\lambda^j}_1>\lambda^j\to\infty$ together with the decay estimate in Lemma \ref{decay} implies that $u_{\lambda^j}(x^j)\to 0$. Therefore, by \eqref{eq:cont1}, we also have $u(x^j)\to 0$. Since $u$ is continuous by Lemma \ref{smooth} and strictly positive by Lemma \ref{decay}, we obtain \eqref{eq:cont1inf}. Hence the bound \eqref{eq:cont1bdd} follows from \eqref{eq:decrease} and \eqref{eq:cont1}.

Now choose $j$ sufficiently large such that $\lambda^j>\overline\lambda$ holds. Then \eqref{eq:cont1bdd} implies that $\overline\lambda<(x^j)^{\overline\lambda}_1< (x^j)^{\lambda^j}_1$. Thus, by \eqref{eq:decrease}, we conclude
\begin{equation}
\label{eq:cont1est}
u_{\overline\lambda}(x^j)> u_{\lambda^j}(x^j) \,.
\end{equation}

On the other hand, \eqref{eq:cont1inf}, \eqref{eq:cont1bdd} and \eqref{eq:com} together with Lemma \ref{asymp} (iii) (and $\overline\lambda$ instead of $\lambda^j$) imply that
$$
\frac{|x^j|^6}{2(\overline\lambda -x_1^j)} \left(u(x^j)-u_{\overline\lambda}(x^j)\right) \to \frac{4\overline\lambda}{\pi^2} \int f(y) \,dy >0 \,,
$$
contradicting \eqref{eq:cont1} and \eqref{eq:cont1est}. Hence the set $\Lambda$ is non-empty.


\emph{Step 2. $\lambda_1:=\inf\Lambda=0$.}

Again, we argue by contradiction and assume that $\lambda_1>0$. We note that $u(x)\geq u_\lambda(x)$ for all $x$ with $x_1<\lambda$ and all $\lambda>\lambda_1$. Hence, by continuity (see Lemma \ref{smooth}), we also have $u(x)\geq u_{\lambda_1}(x)$ if $x_1<\lambda_1$. Note that the function $w:= u-u_{\lambda_1}$ satisfies the equation
$$
\sqrt{-\Delta} w + Vw=-w +f\,,
\qquad
V := -\frac12 (u^2+u_{\lambda_1}^2)*|x|^{-1}
$$
with inhomogeneity
$$
f(x) := \frac12 (u(x)+u_{\lambda_1}(x)) \left( (u^2-u_{\lambda_1}^2)*|x|^{-1} \right)(x) .
$$
Next, a calculation shows that
\begin{multline*}
f(x) = \frac12 (u(x)+u_{\lambda_1}(x))  \int_{y_1<\lambda_1} \left( \frac1{|x-y|} \right . - \\  \left . \frac1{\sqrt{(x_1+y_1-2\lambda_1)^2 +|x'-y'|^2}} \right) \left( u(y)^2-u_{\lambda_1}(y)^2 \right) \,dy \,.
\end{multline*}
Since $|x-y|<\sqrt{(x_1+y_1-2\lambda)^2 +|x'-y'|^2}$ if $x_1,y_1<\lambda_1$, we see that $f\geq 0$ in $\{x_1<\lambda_1\}$ and hence $\sqrt{-\Delta} w \geq-w-Vw\geq -w$ in that set. Moreover, recall that $w = u - u_{\lambda_1}$ belongs to all $H^s(\R^3)$. Therefore, by the non-local Hopf lemma (Proposition \ref{hopf}), we either have $w\equiv 0$, or else
\begin{equation}
\label{eq:cont2hopf}
w>0
\ \text{in}\ \{x_1<\lambda_1\}
\qquad
\text{and}
\qquad
\frac{\partial w}{\partial x_1}(x) < 0
\ \text{on}\ \{x_1=\lambda_1\} \,.
\end{equation}
The first case cannot occur since $w\equiv 0$, $\lambda_1>0$ and \eqref{eq:com} imply $u\equiv 0$, but for $u\equiv 0$ one has $\lambda_1=0$.

Hence we will assume that \eqref{eq:cont2hopf} holds. By definition of $\lambda_1$, there exist sequences of numbers $(\lambda^j)\to\lambda_1-$ and points $(x^j)$ with $x^j_1<\lambda^j$ such that
\begin{equation}
\label{eq:cont2}
u(x^j)<u_{\lambda^j}(x^j) \,.
\end{equation}
Passing to a subsequence if necessary, we may either assume that $x^j\to x$ or else that $|x^j|\to\infty$.

If $x_j\to x$, then \eqref{eq:cont2} implies $u(x)\leq u_{\lambda_1}(x)$. Moreover, since $x_1\leq\lambda_1$, the first relation in \eqref{eq:cont2hopf} allows us to deduce that $x_1=\lambda_1$ and $u(x)= u_{\lambda_1}(x)$. Now \eqref{eq:cont2} yields $\frac{\partial u}{\partial x_1}(x)\geq 0$, contradicting the second relation in \eqref{eq:cont2hopf}.

If $|x^j|\to\infty$, then we argue as in Step 1 (using Lemma \ref{asymp} (iii) with the sequence $(\lambda^j)$) to arrive at a contradiction.


\emph{Step 3. Conclusion.}

In the previous step we have shown that $u(x)\geq u_\lambda(x)$ if $x_1<\lambda$ for any $\lambda>0$. Hence by continuity $u(x)\geq u(-x_1,x')$ if $x_1<0$. Repeating the same argument with $u(x)$ replaced by $u(-x_1,x')$ (and noting that the choice of the origin in \eqref{eq:com} is not affected by this replacement) yields the reverse inequality $u(-x_1,x')\geq u(x)$ if $x_1<0$. Hence $u(\cdot,x')$ is symmetric with respect to $x_1=0$. Using the nonlocal Hopf lemma (Proposition \ref{hopf}) as in Step 2, we find that if $\lambda>0$ then $u(x)> u_\lambda(x)$ for all $x_1<\lambda$. This means that $u(\cdot,x')$ is strictly decreasing for $x_1>0$. The proof of Proposition \ref{symm1} is complete. \hfill $\blacksquare$

\section{Real analyticity}\label{sec:anal}

In this section, we prove that any real-valued solution of the equation \eqref{eq:eq1} is real-analytic, which is a substantial improvement of Lemma \ref{smooth} above. Our proof will derive pointwise exponential decay in Fourier space. A similar argument has been applied in the analyticity proof of solitary waves for some nonlinear water wave equations in $d=1$ spatial dimension; see \cite{LiBo}. However, apart from higher dimensionality, our case also involves a nonlocal nonlinearity. To deal with this difficulty of a nonlocal nonlinearity, we derive exponential bounds in Fourier space for a coupled system of equations.

Our main result on analyticity is as follows.
\begin{theorem}\label{anal}
Let $u\in H^{1/2}(\R^3)$ be a real-valued solution of \eqref{eq:eq1}. Then there exists a constant $\sigma>0$ and an analytic function $\tilde u$ on $\{z\in\C^3 :\ |\im z_j|<\sigma, 1\leq j\leq 3 \}$ such that $\tilde u(x)=u(x)$ if $x\in\R^3$.
\end{theorem}

Note that we do not assume $u$ to be non-negative. Moreover, our proof is independent of the radial symmetry established in the Section \ref{sec:symmetry}. We follow the technique developed in \cite{LiBo}. The heart of the argument is contained in the following statement.

\begin{proposition}\label{anallemma}
Let $\lambda, \alpha >0$, $0\leq f\in L_1(\R^3)$ and $0\leq W\in L_1(\R^3, (1+|\xi|)d\xi)$ such that
\begin{equation}\label{eq:anallemma}
(|\xi|+\lambda)f \leq W * f \,,
\qquad
|\xi|^2 W \leq \alpha f * f \,.
\end{equation}
Then there exist non-negative functions $g_n$, $n\in\N_0$, and constants $a,b>0$ such that
\begin{equation}
\label{eq:analind}
|\xi|^n f \leq g_n * f \,,
\qquad
\|g_n\|_1 \leq a b^{n} (2n+1)^{n-1} \,.
\end{equation}
In particular, if $f\in L_p(\R^3)$ for some $1\leq p \leq \infty$, then
\begin{equation}
\label{eq:analp}
\| |\xi|^n f \|_p \leq a b^{n} (2n+1)^{n-1} \|f\|_p \,.
\end{equation}
\end{proposition}

At several places in this proof we will use the so-called Abel identity,
\begin{equation}\label{eq:abel}
\sum_{l=0}^n \binom{n}{l} (l+a)^{l-1} (n-l+b)^{n-l-1} = \frac{a+b}{ab} (n+a+b)^{n-1} \,,
\end{equation}
see \cite[p. 18]{Ri}.

\begin{proof}
We prove \eqref{eq:analind} by induction over $n$. For $n=0$, \eqref{eq:analind} follows from \eqref{eq:anallemma} with $g_0:=\lambda^{-1} W$ and any $a \geq \lambda^{-1} \|W\|_1$. Now let $n\geq 1$ and assume that \eqref{eq:analind} has already been shown for all smaller values of $n$. By the triangle inequality one has
$|\xi|^{n-1} \leq \sum_{l=0}^{n-1} \binom{n-1}{l} |\xi-\eta|^l |\eta|^{n-1}$ and therefore by \eqref{eq:anallemma} and the induction hypothesis
$$
|\xi|^n f \leq \sum_{l=0}^{n-1} \binom{n-1}{l} \left( |\eta|^l W\right) * \left( |\eta|^{n-1-l} f\right)
\leq g_n * f
$$
where
$$
g_n := \sum_{l=0}^{n-1} \binom{n-1}{l} \left( |\eta|^l W\right) * g_{n-1-l} \,.
$$
Hence
\begin{equation}
\label{eq:gnnorm}
\| g_n \|_1 \leq \sum_{l=0}^{n-1} \binom{n-1}{l} \| |\eta|^l W\|_1 \|g_{n-1-l} \|_1 \,.
\end{equation}
Next, we estimate $\| |\eta|^l W\|_1$ for $l\geq 2$. For $m\in\N_0$ one has again by \eqref{eq:anallemma}, the triangle inequality and the induction hypothesis for $m<n$
$$
|\xi|^{m+2} W \leq \alpha \sum_{k=0}^m \binom{m}{k} \left(|\eta|^k f \right) * \left(|\eta|^{m-k} f \right)
\leq \alpha \sum_{k=0}^m \binom{m}{k} g_k*f*g_{m-k}*f \,.
$$
Hence
\begin{align*}
\| |\xi|^{m+2} W \|_1 & \leq \alpha \|f\|_1^2 \sum_{k=0}^m \binom{m}{k} \|g_k\|_1 \|g_{m-k}\|_1 \\
& \leq \alpha a^2 b^{m} \|f\|_1^2 \sum_{k=0}^m \binom{m}{k}  (2k+1)^{k-1} (2(m-k)+1)^{m-k-1} \\
& = 2 \alpha a^2 b^{m} \|f\|_1^2 (2m+2)^{m-1} \,,
\end{align*}
where we used Abel's identity \eqref{eq:abel} in the last calculation. In order to simplify some arithmetics below, we estimate this by
\begin{align}\label{eq:wdecay}
\| |\xi|^l W \|_1 & \leq 2 \alpha a^2 b^{l-2} \|f\|_1^2 (2l+2)^{l-1}
\end{align}
for $l\geq 2$. If we choose $a^2 b^{l-2}$ large enough, then this holds also for $l=0$ and $l=1$.

Plugging this into \eqref{eq:gnnorm} and using the induction hypothesis and again Abel's identity, we arrive at
\begin{align*}
\| g_n \|_1 & \leq 2 \alpha a^3 b^{n-3} \|f\|_1^2 \sum_{l=0}^{n-1} \binom{n-1}{l} (2l+2)^{l-1} (2(n-1-l)+1)^{n-l-2} \\
& = 3 \alpha a^3 b^{n-3} \|f\|_1^2 (2n+1)^{n-2} \,.
\end{align*}
This proves the assertion provided we have
\begin{equation}\label{eq:abchoice}
3 \alpha a^2 \|f\|_1^2
\leq b^{3} (2n+1) \,.
\end{equation}
Let us show that such a choice of parameters $a$ and $b$ is possible. We fix the ratio $a/b$ by $a^2/b^2=\|W\|_1/(\alpha\|f\|_1^2)=:c^2$ with $c>0$, so that \eqref{eq:wdecay} holds for $l=0$. Now we choose $a$ (keeping the ratio $a^2/b^4$ fixed) to be
$$
a := \max\{ \lambda^{-1} \|W\|_1,\, \||\xi|W\|_1/(2\alpha c\|f\|_1^2), \, \alpha c^3 \|f\|_1^2 \} \,.
$$
Hence \eqref{eq:analind} holds for $n=0$, \eqref{eq:wdecay} holds for $l=1$, \eqref{eq:abchoice} holds for all $n\geq 1$ and the proof is complete.
\end{proof}

\begin{proof}[Proof of Theorem \ref{anal}]
Let $u\in H^{1/2}$ be a real-valued solution of \eqref{eq:eq1} and $\hat u$ its Fourier transform \eqref{eq:fourier}. Then
$$
|\xi| \hat u - w*\hat u=-\hat u
$$
with
$$
w(\xi) := \frac{1}{(2\pi)^3} \iint_{\R^3\times\R^3} \frac{u(y)^2 e^{-i\xi\cdot x}}{|x-y|} \,dx\,dy
= \frac{1}{|\xi|^{2} 2 \pi^2} \int_{\R^3} u(y)^2 e^{-i\xi\cdot y} \,dy
= \frac{\hat u * \hat u (\xi)}{2\pi^2 |\xi|^2 } \,.
$$
Here we used that $u$ is real-valued. Hence $f:=|\hat u|$ satisfies \eqref{eq:anallemma} with $W:=|w|$, $\alpha= (2\pi^2)^{-1}$ and $\lambda=1$.

We claim that the assumptions of Lemma \ref{anallemma} are satisfied. Indeed, by Lemma \ref{decay}, we have $u\in L_1$ and hence $f\in L_\infty$. Also, by Lemma \ref{smooth}, we conclude $f\in L_1$. This implies that $\hat u *\hat u\in L_1\cap L_\infty$ and hence $W\in L_1(\R^3, (1+|\xi|)d\xi)$. Therefore we can apply Lemma \ref{anallemma} and obtain constants $a$ and $b$ such that
$$
\sup_{\xi} \left| \exp(\tau |\xi|) \hat u(\xi) \right| 
\leq \sum_n \frac{\tau^n}{n!} \||\xi|^n f\|_\infty
\leq \sum_n \alpha_n \tau^n
$$
with $\alpha_n := (n!)^{-1} a b^{n} (2n+1)^{n-1} \|f\|_\infty$. Since we find
$$
\frac{\alpha_{n+1}}{\alpha_n}
= \frac{b (2n+3)^{n} }{ (n+1) (2n+1)^{n-1}}
\to 2be \, ,
$$
the above supremum is finite for $\tau<\sigma:=(2be)^{-1}$. Thus the function
$$
\tilde u(z) := (2\pi)^{-3/2} \int_{\R^3} e^{i(x_1 z_1+x_2 z_2+x_3 z_3)} \hat u(\xi) \,d\xi
$$
is analytic in $\{z\in\C^3 :\ |\im z_j|<\sigma, 1\leq j\leq 3 \}$ and it coincides with $u$ on $\R^3$ by Plancherel's theorem.
\end{proof}

\begin{remark} \label{rem:Qmanal}
As a further corollary of Proposition \ref{anallemma}, we note that real-analycity also follows for real-valued solutions $u \in H^{1/2}(\R^3)$ satisfying
$$
\sqrt{-\Delta + m^2} \, u - \big ( u^2 \ast |x|^{-1} \big ) u =  -\mu u,
$$
where $m>0$ and $\mu > -m$ are given parameters.
\end{remark}

\section{Real analyticity II} \label{sec:anal2}

In this section, we establish (as some additional result) analyticity of kernel elements of the linearized operator associated with $Q$ solving \eqref{eq:eq}. Although the arguments will follow closely Section \ref{sec:anal}, we provide the details of the (tedious) adaptation.

Suppose that $Q \in H^{1/2}(\R^3)$ is a real-valued solution to \eqref{eq:eq1}. According to Theorem \ref{anal}, $Q$ is real-analytic. We consider the associated linearized operator
$$
L_+ \xi = \sqrt{-\Delta} \, \xi + \xi - \big ( Q^2 \ast |x|^{-1} \big ) \xi - 2 Q \big ( (Q \xi) \ast |x|^{-1} \big ).
$$
This defines a self-adjoint operator in $L^2(\R^3)$ with operator domain $H^1(\R^3)$.

We have the following result.
\begin{proposition}\label{anal2}
If $v\in\ker L_+$, then $v$ is real-analytic. More precisely, there exists a constant $\sigma$ and an analytic function $\tilde v$ on $\{z\in\C^3 : |\im z_j|<\sigma, \, 1\leq j\leq 3 \}$ such that $\tilde v(x)=v(x)$ if $x\in\R^3$.
\end{proposition}

The proof of Proposition \ref{anal2} will be given at the end of the section. First, we establish the following auxiliary fact.

\begin{lemma}\label{nondegdecay}
Assume that $v\in\ker L_+$ is radial. Then $v\in L_1$ and hence $\hat v\in L_\infty$.
\end{lemma}

\begin{proof}
We have
$$
v=(\sqrt{-\Delta}+1)^{-1} (f_1 + f_2)
$$
where $f_1:= (Q^2*|x|^{-1}) v$ and $f_2:=2((Qv)*|x|^{-1})Q$. Let us first consider $f_2$ and we note that
\begin{equation*}
|f_2(x)| \leq C (1+|x|)^{-4} | (Qv) \ast |x|^{-1}|(x) \leq C (1 + |x|)^{-5},
\end{equation*} 
Here, the pointwise bound on $Q$ comes from Lemma \ref{decay} and the pointwise bound on $(Qv) \ast |x|^{-1}$ follows from combining Hardy's inequality (to get an $L_\infty$-bound) and Newton's theorem to conclude that $|(Qv) \ast |x|^{-1}|(x) \leq C/|x|$ for $|x| >0$.

Next, we note that $f_1\in L_{6/5+}$, since $Q^2*|x|^{-1}\in L_{3+}$ and $v\in L_2$. Since $(\sqrt{-\Delta}+1)^{-1}$ is the convolution by a function in $L_1$ (indeed, a function bounded by a constant times $\min\{|x|^{-2},|x|^{-4}\}$), we conclude that $v\in L_{6/5+}$. But this implies that $f_1\in L_1$, and hence $v$ is the convolution of an $L_1$ kernel with the $L_1$ function $f_1+f_2$, and hence $v$ must be in $L_1$.
\end{proof}

As a next step and similarly as in Section \ref{sec:anal}, we prove the following statement.

\begin{lemma}\label{analker}
Let $W$ and $V$ be non-negative functions on $\R^n$ satisfying for some $a,b>0$ and all $n\in\N_0$ and $1\leq p\leq \infty$
\begin{equation}\label{eq:analkerass}
\| |\xi|^n W \|_1 \leq ab^{n} (2n+1)^{n-1} \,,
\qquad
\| |\xi|^n V \|_p \leq ab^{n} (2n+1)^{n-1} \,.
\end{equation}
Let $\lambda >0$, $0\leq f\in L_2(\R^3)\cap L_\infty(\R^3)$ and $g\geq 0$ measurable such that
\begin{equation}\label{eq:analker}
\begin{split}
(|\xi|+\lambda)f \leq W * f + V*g \,,
\qquad
|\xi|^2 g \leq V*f \,.
\end{split}
\end{equation}
Then there exist $\tilde a,\tilde b>0$ such that for all $n\in\N_0$,
\begin{equation}
\label{eq:analkerind}
\| |\xi|^n f \|_\infty \leq \tilde a \tilde b^{n} (2n+1)^{n-1} \,,
\qquad
\| |\xi|^{n+2} g \|_\infty \leq \tilde a \tilde b^{n+2} (2(n+2)+1)^{(n+2)-1} \,.
\end{equation}
\end{lemma}

\begin{proof}
We begin by showing that $g$ and $|\xi|g$ are integrable and that $|\xi|^2 g$ is bounded. To see this, note that since $V\in L_1\cap L_2$ and $f\in L_2$, $h:=|\xi|^2 g \leq V*f \in L_2 \cap L_\infty$ and therefore
$$
\int_{\R^3} g \,d\xi \leq \|h\|_\infty \int_{|\xi|<1} |\xi|^{-2} \,d\xi  + \|h\|_2 \left( \int_{|\xi|>1} |\xi|^{-4} \,d\xi \right)^{1/2} <\infty \,.
$$
Using this information, as well as $W\in L_1$, $f\in L_2$, $V\in L_2$, we find $|\xi| f \leq W * f + V*g \in L_2$. By the triangle inequality, $|\xi| h \leq |\xi|(V*f) \leq (|\eta|V)*f + V*(|\eta| f) \in L_2\cap L_\infty$, and therefore
$$
\int_{\R^3} |\xi| g \,d\xi \leq \|h\|_\infty \int_{|\xi|<1} |\xi|^{-1} \,d\xi  + \||\xi| h\|_2 \left( \int_{|\xi|>1} |\xi|^{-4} \,d\xi \right)^{1/2} <\infty \,.
$$

We define
$$
\tilde a := \max\{\|f\|_\infty , \| g \|_1 \} \,.
$$
Since \eqref{eq:analkerass} remains true if $b$ is increased, we may assume that
$$
b \geq \max\{ \left(\| |\xi|^2 g \|_\infty/ (5 \tilde a) \right)^{1/2}, \| |\xi| g \|_1 / \tilde a , 2\tilde a, \left( 2\tilde a \right)^{1/2}/7  \}
$$
Note that these choices imply that
\begin{equation}\label{eq:g1est}
\| |\xi|^{l} g \|_1 \leq \tilde a b^{l} (2l+1)^{l-1} 
\qquad \text{for } l=0,1 \,.
\end{equation}

Having modified $b$ in this way, we shall prove \eqref{eq:analkerind} with $\tilde b=b$ (and $\tilde a$ as defined above). We proceed by induction with respect to $n\in\N_0$. For $n=0$ the assertion is an immediate consequence of our choices for $\tilde a$ and $b$. Now let $n\geq 1$ and assume that \eqref{eq:analkerind} has already been shown for all smaller values of $n$.  By \eqref{eq:analker} and the triangle inequality
$$
|\xi|^n f \leq \sum_{l=0}^{n-1} \binom{n-1}{l} \left( \left( |\eta|^l W \right) * \left(|\eta|^{n-1-l} f \right) + \left(|\eta|^l V \right) * \left(|\eta|^{n-1-l} g \right) \right)
$$
and therefore
\begin{align*}
\| |\xi|^n f \|_\infty \leq \sum_{l=0}^{n-1} \binom{n-1}{l} \| |\xi|^l W \|_1 \||\xi|^{n-1-l} f \|_\infty 
& + \sum_{l=0}^{n-3} \binom{n-1}{l} \| |\xi|^l V \|_1 \| \|\xi|^{n-1-l} g \|_\infty \\
& + \sum_{l=n-2}^{n-1} \binom{n-1}{l} \| |\xi|^l V \|_\infty \| \|\xi|^{n-1-l} g \|_1 \,.
\end{align*}
(The middle sum should be discarded if $n\leq 2$.) Hence by the induction hypothesis, by \eqref{eq:analkerass} and \eqref{eq:g1est}
$$
\| |\xi|^n f \|_\infty \leq 2a \tilde a b^{n-1} \sum_{l=0}^{n-1} \binom{n-1}{l} (2l+1)^{l-1} (2(n-l-1)+1)^{n-l-2} 
= 4 a \tilde a b^{n-1} (2n)^{n-2}
\,.
$$
In the last calculation we used Abel's identity \eqref{eq:abel}. This proves the first assertion in \eqref{eq:analkerind}, provided we have
$$
4 a \tilde a b^{n-1} (2n)^{n-2} \leq a b^n (2n+1)^{n-1}
$$
for all $n\geq 1$. It is easy to see that this holds if $2\tilde a\leq b$, which holds by the choice of $b$.

To prove the second assertion in \eqref{eq:analkerind} we proceed similarly as before, using the triangle inequality to get
$$
|\xi|^{n+2} g \leq \sum_{l=0}^{n} \binom{n}{l} \left( |\eta|^l V \right) * \left(|\eta|^{n-l} f \right)
$$
and hence
$$
\| |\xi|^{n+2} g \|_\infty
\leq \sum_{l=0}^{n} \binom{n}{l} \| |\eta|^l V \|_1  \| |\eta|^{n-l} f \|_\infty \,.
$$
Relation \eqref{eq:analkerass} together with the estimates for $|\xi|^{n-l} f$ (which we have already proved) and Abel's identity \eqref{eq:abel} imply that
$$
\| |\xi|^{n+2} g \|_\infty
\leq a \tilde a b^{n} \sum_{l=0}^{n} \binom{n}{l} (2l+1)^{l-1} (2(n-l) +1)^{n-l-1}
= 2 a \tilde a b^{n} (2n+2 )^{n-1} \,.
$$
This proves the second assertion in \eqref{eq:analkerind}, provided we have
$$
2 a \tilde a b^{n} ( 2n+2 )^{n-1} \leq a b^{n+2} (2(n+2)+1)^{(n+2)-1}
$$
for all $n\geq 1$. It is easy to see that this holds if $\tilde a\leq 49 b^2/2$, which holds by the choice of $b$. This completes the proof of Lemma \ref{analker}.
\end{proof}

\begin{proof}[Proof of Proposition \ref{anal2}]
For the Fourier transform $\hat v$, the equation $L_+ v=0$ leads to
$$
|\xi| \hat v - w*\hat v -\tilde w*\hat Q = - \hat v
$$
with $w$ as in the proof of Proposition \ref{anal} and
$$
\tilde w(\xi) :=  \frac{ \hat Q * \hat v (\xi)}{\pi^2 |\xi|^2 } \,.
$$
Hence $f:=|\hat v|$ and $g:=|\tilde w|$ satisfy \eqref{eq:analker} with $W:=|w|$ and $V:= \pi^{-1}|\hat Q|$. (Indeed, we know that $\hat Q>0$, but we do not need this fact.) Now the assumptions \eqref{eq:analkerass} can be deduced from \eqref{eq:analp} and \eqref{eq:wdecay} after modifying $a$ and $b$. (Strictly speaking, we use the equation before \eqref{eq:wdecay} which yields the term $(2l+2)$ in \eqref{eq:wdecay} replaced by $(2l+1)$.) Moreover, $f\in L_\infty$ by Lemma \ref{nondegdecay}. Now the statement of Proposition \ref{anal2} follows as in the proof of Theorem \ref{anal}.
\end{proof}



\bibliographystyle{amsalpha}

\begin{thebibliography}{GiNiNi2}


		



\bibitem[BiLoWa]{BiLoWa} M. Birkner, J. A. L\'opez-Mimbela, A. Wakolbinger, \textit{Comparison results and steady states for the {F}ujita equation with fractional {L}aplacian}. Ann. Inst. H. Poincar\'e Anal. Non Lin\'eaire \textbf{22} (2005), no. 1, 83--97.



\bibitem[BrLiLu]{BrLiLu} H. J. Brascamp, E. H. Lieb, J. M. Luttinger, \textit{A general rearrangement inequality for multiple integrals,} J. Functional Analysis \textbf{17} (1974), 227--237.



\bibitem[CaMaSi]{CaMaSi} R. Carmona, W. C. Masters, B. Simon, \textit{Relativistic {S}chr\"odinger operators: asymptotic behavior of the eigenfunctions}, J. Funct. Anal. \textbf{91} (1990), no. 1, 117--142.



\bibitem[ChLiOu]{ChLiOu} W. Chen, C. Li, B. Ou, \textit{Classification for solutions of an integral equation}. Comm. Pure Appl. Math. \textbf{59} (2006), no. 3, 330-343.



\bibitem[ElSc]{ElSc} A. Elgart, B. Schlein, \textit{Mean field dynamics of boson stars}, Comm. Pure Appl. Math. \textbf{60} (2007), no. 4, 500--545.

\bibitem[Ev]{Ev} L. C. Evans, \textit{Partial differential equations}, Graduate Studies in Mathematics \textbf{19}, American Mathematical Society, Providence, RI, 1998.

\bibitem[FrSe]{FrSe} R. L. Frank, R. Seiringer, \textit{Non-linear ground state representations and sharp Hardy inequalities}. J. Funct. Anal. \textbf{255} (2008), 3407--3430.

\bibitem[FraLe]{FraLe} R. L. Frank, E. Lenzmann, \textit{Uniqueness and nondegeneracy of ground states for $(-\Delta)^s Q + Q - Q^{\alpha+1} =0$ in $\mathbb{R}$}. Preprint available at {\tt arXiv:1009.4042}.

\bibitem[FrLe]{FrLe} J.~Fr\"ohlich, E. Lenzmann, \textit{Blowup for nonlinear wave equations describing boson stars}, Comm. Pure Appl. Math. \textbf{60} (2007), no. 11, 1691--1705.

\bibitem[GiNiNi]{GiNiNi} B. Gidas, W. M. Ni, L. Nirenberg, \textit{Symmetry and related properties via the maximum principle}, Comm. Math. Phys. \textbf{68} (1979), no. 3, 209--243.

\bibitem[GiNiNi2]{GiNiNi2}  B. Gidas, W. M. Ni, L. Nirenberg, \textit{Symmetry of positive solutions of nonlinear elliptic equations in $R\sp{n}$}. In: Mathematical analysis and applications, Part A,  pp. 369--402, Adv. in Math. Suppl. Stud., 7a, Academic Press, New York-London, 1981.

\bibitem[He]{He} I. W. Herbst, \textit{Spectral theory of the operator $(p^2+m^{2})^{1/2}-Ze^{2}/r$}. Comm. Math. Phys. \textbf{53} (1977), no. 3, 285--294.




\bibitem[Le2]{Le2} E. Lenzmann, \textit{Well-posedness for semi-relativistic {H}artree equations of critical type}, Math. Phys. Anal. Geom. \textbf{10} (2007), no. 1, 43--64.

\bibitem[Le3]{Le3} E. Lenzmann, \textit{Nonlinear dispersive equations describing Boson stars}, ETH Dissertation No.~16572 (2006).

\bibitem[LiBo]{LiBo} Y. A. Li, J. L. Bona, \textit{Analyticity of solitary-wave solutions of model equations for long waves}. SIAM J. Math. Anal. \textbf{27} (1996), no. 3, 725--737.


\bibitem[Li]{Li} E. H. Lieb, \textit{Existence and uniqueness of the minimizing solution of Choquard's nonlinear equation}. Studies in Appl. Math. \textbf{57} (1977), no. 2, 93--105.

\bibitem[LiLo]{LiLo} E. H. Lieb, M. Loss, \textit{Analysis. Second edition}. Graduate Studies in Mathematics \textbf{14}, American Mathematical Society, Providence, RI, 2001.

\bibitem[LiOx]{LiOx} E. H. Lieb, S. Oxford, \textit{An improved lower bound on the indirect Coulomb energy}. Int. J. Quant. Chem. \textbf{19} (1981), 427--439.

\bibitem[LiYa]{LiYa} E. H. Lieb, H.-T. Yau, \textit{The Chandrasekhar theory of stellar collapse as the limit of quantum mechanics}. Comm. Math. Phys. \textbf{112} (1987), no. 1, 147--174.

\bibitem[MaZh]{MaZh} L. Ma, L. Zhao, \textit{Classification of positive solitary solutions of the nonlinear Choquard equation}. To appear in Arch. Rat. Mech. (2009). 




\bibitem[Ri]{Ri} J. Riordan, \textit{Combinatorial identities}, John Wiley \& Sons, New York, 1968.


\bibitem[Th]{Th} B. Thaller, \textit{The Dirac equation}. Texts and Monographs in Physics. Springer, Berlin, 1992.


\end{thebibliography}

\end{document}